\documentclass{amsart}
\usepackage[utf8]{inputenc}

\usepackage{amsthm}
\usepackage{amsmath}
\usepackage{amsfonts}
\usepackage{amssymb}
\usepackage{fullpage}
\usepackage{xcolor}
\usepackage{graphicx}
\usepackage{hyperref}
\usepackage{float}
\usepackage{algorithm}
\usepackage{algpseudocode}
\usepackage{tablefootnote}
\usepackage{caption}
\usepackage{subcaption}

\usepackage{hyperref,cleveref,color,verbatim}

\newtheorem{theorem}{Theorem}[section]

\newtheorem{lemma}[theorem]{Lemma}
\newtheorem{cor}[theorem]{Corollary}

\newtheorem{question}[theorem]{Question}

\newtheorem*{theorem*}{Theorem}

\theoremstyle{definition}
\newtheorem{remark}[theorem]{Remark}
\newtheorem{defi}[theorem]{Definition}

\newcommand{\R}{\mathbb{R}}

\newcommand{\EE}{\mathbb{E}}

\newcommand{\tr}{\textup{tr}}

\newcommand{\FW}{\mathcal{F}\mathcal{W}}
\renewcommand{\P}{\mathcal{P}}
\newcommand{\M}{\mathcal{M}}

\DeclareMathOperator*{\Diag}{Diag}
\DeclareMathOperator*{\diag}{diag}

\DeclareMathOperator*{\supp}{supp}
\newcommand{\st}{{\text{ s.t. }}}

\DeclareMathOperator{\conv}{\operatorname{conv}}

\newcommand*{\Sym}{\R^{n \times n}_{\mathrm{sym}}}

\DeclareMathOperator{\argmax}{argmax}

\title{Quadratic Programs with Sparsity Constraints Via Polynomial Roots}
\date{\today}

\author{Kevin Shu}
\address[Kevin Shu]{School of Mathematics, Georgia Institute of Technology, 686 Cherry Street, Atlanta, GA 30332, USA}
\email{kshu8@gatech.edu}

\thanks{We would like to thank Greg Blekherman, Santanu Dey, and Shengding Sun for many insightful conversations. Much of the computational experiments supporting this paper were conducted at the Max Planck Institute for Mathematics in the Sciences.}

\begin{document}

\begin{abstract}
    Quadratic constrainted quadratic programs (QCQPs) are an expressive family of optimization problems that occur naturally in many applications.
    It is often of interest to seek out \emph{sparse} solutions, where many of the entries of the solution are zero.
    This paper will consider QCQPs with a single linear constraint, together with a sparsity constraint that requires that the set of nonzero entries of a solution be small.
    This problem class includes many fundamental problems of interest, such as sparse versions of linear regression and principal component analysis, which are both known to be very hard to approximate.
    We introduce a family of tractible approximations of such sparse QCQPs using the roots of polynomials which can be expressed as linear combinations of principal minors of a matrix.
    These polynomials arose naturally from the study of hyperbolic polynomials.
    Our main contributions are formulations of these approximations and computational methods for finding good solutions to a sparse QCQP.
    We will also give numerical evidence that these methods can be effective on practical problems.
\end{abstract}

\maketitle

\section{Introduction}
\subsection{Problem Setup}
The main objects of interest in this paper are homogeneous \emph{sparse quadratically constrainted quadratic programs} (sparse QCQPs) with a single linear constraint. Let $k$ be a nonnegative integer. A sparse QCQP is an optimization problem of the form
\begin{equation}\label{eq:sparse_qcqp_orig}
    \begin{aligned}
        \max\quad & x^{\intercal}A_0x\\
        \st & x^{\intercal}A_1x = 1\\
            & x \in \R^n\\
            &|\supp(x)| \le k.
    \end{aligned}
    \tag{QCQP}
\end{equation}
Here, $\supp(x) = \{i \in [n] : x_i \neq 0\}$, and $A_1, A_0 \in \Sym$.
While our techniques can produce results for sparse QCQPs with more constraints, we will focus on the one constraint case where $A_1$ is is positive definite, as it contains the problems which are of greatest interest, and also our results are especailly effective in this setting.

Two problems of this form which we will consider in detail are the \emph{sparse linear regression} and \emph{sparse maximum eigenvalue} (sparse PCA from here on) problems.
These are the natural sparse versions of the corresponding classical linear algebra problems, for which formulations as sparse QCQPs can be found in \Cref{sec:sparseReg} and \Cref{sec:sparsePCA}.
These are both problems which have been studied extensively in data science as methods for extracting features from data, for example in \cite{d2008optimal,candes2005decoding}.
Both sparse linear regression and sparse PCA are known to be NP hard \cite{welch1982algorithmic, magdon2017np}, and as such, sparse QCQPs with a single constraint in general are NP-hard to solve.

There is an extensive literature on the various types of sparse QCQPs, mostly related to the sparse linear regression and sparse PCA problems.
For this reason, and because our goal is mostly to highlight a concrete application of our algbraic methods, we will not provide an exaustive list of previous work on these subjects.
We will make particular note of approaches which increase sparsity using a penalty function such as the $\ell_1$ norm to reward solutions for being more sparse.
Such methods constitute the state of the art for sparse optimization both in theory and in practice \cite{tibshirani1996regression, hastie2020best, zou2006sparse}.
Under certain conditions, these methods can be shown to find the global optimum for \Cref{eq:sparse_qcqp_orig}, but generally, these are heuristic methods which find good solutions to the original sparse regression problem.

There have also been approaches which certify that their solutions are optimal or close to optimal, such as \cite{bertsimas2016best, bertsimas2022solving}, which use techniques such as mixed integer optimization and branch and bound.
These methods do not have polynomial time runtime, though they can be surprisingly effective on small instances of these problems.

Our method perhaps most strongly resembles greedy iterative methods for solving sparse linear regression, such as orthogonal matching pursuit \cite{tropp2004greed}.
Firstly, we observe that if the support of a solution is fixed in advance, then \Cref{eq:sparse_qcqp_orig} reduces to a simple QCQP with a single constraint.
Such single constraint QCQPs are easily solved, so the difficult aspect is finding the combinatorial problem of finding the support of a good solution.
Our method builds up the support of a solution iteratively, i.e. we start with an empty set for the support, and then repeatedly add elements to the support in such a way as to maximize some heuristic score for that set.

Our methods are heuristic, and inspired by interior point method for solving semidefinite programs \cite{alizadeh1995interior, ben2001lectures}.
In the case of semidefinite programming, one key fact that is used for interior point methods is that the determinant function is zero on the boundary of the positive semidefinite cone.
This fact leads to a connection between the optimum value of a semidefinite program and the zero set of the determinant function.
Concretely, if we consider the one-constraint semidefinite program
\begin{equation*}
    \begin{aligned}
        \max\quad & \tr(A_0X)\\
        \st & \tr(A_1X) = 1\\
            & X \succeq 0,
    \end{aligned}
\end{equation*}
where $A_1$ is PSD, then the optimum value of this program is equal to the maximum zero of the univariate polynomial $g(t) = \det(A_1 t - A_0)$.
This can be seen by considering the dual semidefinite program and applying the fact that the determinant must be zero on the boundary of the PSD cone.
We will modify these facts in the sparse setting by introducing sparse versions of the determinant, and applying them to solving sparse versions of QCQPs.

\begin{defi}
We say that a polynomial $p$ in a symmetric matrix of indeterminants is a \emph{linear combination of principal minors (LPM)} if it is not identically zero and it is of the form
\[
    p(X) = \sum_{S \subseteq [n] : |S| = k} a_S\det(X|_S),
\]
for some coefficients $a_S$, and where $X|_S$ denotes the principal submatrix of $X$ indexed by the set $S$.
\end{defi}
LPM polynomials were studied in \cite{blekherman2021linear} for their connection to hyperbolic polynomials and convex optimization.
In order to apply our methods, we will need as input an oracle that can efficently compute an LPM polynomial $p$ where all of the $a_S$'s are strictly positive.

The key examples of such polynomials is are the \emph{characteristic coefficients}:
\[
    c_n^k(X) = \sum_{S \subseteq [n] : |S| = k} \det(X|_S).
\]
The characteristic coefficients are so named because they are the coefficients of the characteristic polynomial of the matrix $X$ (we will refer to \cite{horn2012matrix} for such linear algebra facts).
We will see that the characteristic coefficients have a number of properties which make them suitable for our methods, and also, it is possible to compute their values efficiently \cite{baer2021faddeev}.

Our main observation is the following: the roots of LPM polynomials can be used to bound the objective value of \Cref{eq:sparse_qcqp_orig}.
For a univariate polynomial $g(t)$, denote by $\eta_g$ the maximum real root of $g$, or $-\infty$ if $g$ has no real roots.
\begin{theorem}
    \label{thm:root_thm}
    Consider a sparse QCQP $\mathcal{Q}$ as in \Cref{eq:sparse_qcqp_orig} so that $A_1$ is positive definite.
    Suppose that $p$ is an LPM polynomial with nonnegative coefficients.
    Then there exists a feasible point $x$ to $\mathcal{Q}$ whose value is at least $\eta_g$, where $g(t) = p(A_1 t - A_0)$.

    That is, the maximum real root of the univariate polynomial $p(A_1 t - A_0)$ is a lower bound for the value of $\mathcal{Q}$.
\end{theorem}
We will discuss an alternative formulation of this theorem in terms of hyperbolic polynomials in \Cref{sec:hyperbolic}, which will allow us obtain results that relax the condition of having only 1 constraint, and also enable us to use non-positive semidefinite values for $A_1$.

We will describe how we intend to make use of this observation here.
\subsection{Contributions}
Our main contribution is to make \Cref{thm:root_thm} effective.
We provide a schematic of an algorithm, \Cref{alg:greedy}, which produces solutions whose values exceed the lower bound provided in \Cref{thm:root_thm} in polynomial time.
Actually, the algorithm we propose typically finds much better solutions than what is guaranteed by the theorem.

The main idea of \Cref{alg:greedy} is to produce from a single LPM polynomial, a sequence of LPM polynomials with a decreasing number of nonzero coefficients, until in the output there is only a single nonzero coefficient.
If the final polynomial we obtain is $p = a_S \det(X|_S)$, then we output the set $S$ for the support of our final solution to \Cref{eq:sparse_qcqp_orig}.
If we can compute each of these LPM polynomials efficiently, and also, the associated root increases in each step, then we can apply \Cref{thm:root_thm} directly to see that our final solution will have objective value at least that guaranteed by \Cref{thm:root_thm}.

While \Cref{alg:greedy} can in principle applied for any LPM polynomial, there are a number of details that need to be considered to make it practical.
In \Cref{sec:char_poly}, we will give a specific implementation of \Cref{alg:greedy} for when the underlying polynomial is the characteristic coefficient $c_n^k$, which runs in roughly $O(n^3)$ time for sparse regression and sparse PCA tasks.
In \Cref{sec:experimental}, we conduct numerical experiments showing that our methods acheive results that are competitive with standard methods for sparse optimization in both approximation quality and in speed.

We also show that these roots lead to natural sparse versions of classical linear algebra theorems.
For example, in classical linear algebra, we can use Cramer's rule to show that for any matrix $A \in \R^{m\times n}$ and any $b \in \R^m$, the least squares regression loss when regressing $b$ against the columns of $A$ is exactly
\[
    \|b\|_2^2 - \frac{\det(A^{\intercal} ( I + bb^{\intercal})A)}{\det(A^{\intercal}A)} + 1.
\]
Similarly, for our sparse method, the upper bound on the sparse least squares regression loss guaranteed by \Cref{thm:root_thm} is given precisely by
\[
    \|b\|_2^2 - \frac{p(A^{\intercal} ( I + bb^{\intercal})A)}{p(A^{\intercal}A)} + 1.
\]
Thus, in this case, we may think of this as a generalized Cramer's rule.

We also show that this quantity appears naturally when we consider the expectation of the least squares regression loss when regressing using a \emph{random} sample of columns of $A$ drawn from the truncated determinantal point process defined by $A^{\intercal}A$ in \Cref{sec:probabilistic}.

In the case of sparse PCA, we have the classical fact that the maximum eigenvalue of a matrix $A$ is given by
\[
    \max \{t : \det(tI - A) = 0\}.
\]
Our methods guarantee that for any LPM polynomial $p$ with nonnegative coefficients,
\[
    \max \{t : p(tI - A) = 0\}
\]
is a lower bound for the maximum $k$-sparse eigenvalue of $A$.

We also show that in principal, there is always some LPM polynomial for which \Cref{thm:root_thm} is exact in \Cref{thm:continuous_formulation}.
This result does lead to an efficient algorithm because it would require knowing a-priori what the support of an optimal solution is.

We conclude by describing how these methods can be generalized to multiple constraints in \Cref{sec:hyperbolic}, where we also describe the connections between these ideas and hyperbolic polynomials.

\subsection{Paper Layout}
The structure of this paper is as follows: the first few sections of the paper are meant as an extended summary of the practical results of our paper.
We first describe a convex formulation of \Cref{eq:sparse_qcqp_orig} and how it naturally leads to a proof of \Cref{thm:root_thm} in \Cref{sec:lpm_relaxation}.
We then define the high level idea of our algorithm in \Cref{subsec:heuristic}, and then give slightly more detailed results for the sparse regression and sparse PCA problems in \Cref{sec:sparseReg} and \Cref{sec:sparsePCA}.
For example, we will discuss our probabilistic formulation of our methods for sparse regression, and how they are related to determinantal point processes, which are popular in machine learning.
We round out the practical discussion of our methods in \Cref{sec:experimental} with numerical results showing that our results are both accurate and fast, in comparison to standard methods for these problems.

After discussing the practical merits of our methods, we give a more detailed description of our efficient algorithm as it applies to the characteristic coefficient in \Cref{sec:char_poly}, which will require some algorithmic ideas for quickly computing the values of characteristic coefficients at a large number of matrices. 
We then give proofs of the main results mentioned earlier in the paper.
We conclude by describing how these results relate to the theory of hyperbolic polynomials and then some open questions.

\section{LPM Relaxations of Sparse QCQPs}
 \label{sec:lpm_relaxation}

We will in fact consider a slightly more general problem than \Cref{eq:sparse_qcqp_orig} that allows for combinatorial constraints on the support of a feasible solution.
Let $\Delta \subseteq 2^{[n]}$ be a family of subsets of $[n]$ where all of the elements of $\Delta$ have size $k$, then a generalized sparse QCQP is an optimization problem of the form
\begin{equation}\label{eq:sparse_qcqp}
    \begin{aligned}
        \max\quad & x^{\intercal}A_0x\\
        \st & x^{\intercal}A_1x = 1\\
            & x \in \R^n\\
            &\supp(x) \in \Delta.
    \end{aligned}
    \tag{QCQP$_{\Delta}$}
\end{equation}
Here, $\supp(x) = \{i \in [n] : x_i \neq 0\}$, and $A_1, A_0 \in \Sym$. The original formulation of a sparse QCQP in \Cref{eq:sparse_qcqp} is the case when $\Delta = \binom{[n]}{k} = \{S \subseteq [n] : |S| = k\}$.

We say that an LPM polynomial $p = \sum a_S \det(X|_S)$ has support $\Delta$ if $\{S : a_S \neq 0\} \subseteq \Delta$.

It will also be useful to think of this as a combinatorial optimization problem like so:
\begin{equation}\label{eq:sparse_qcqp_2}
    \max_{S \in \Delta}
    \begin{aligned}
        \max\quad & x^{\intercal}A_0x\\
        \st & x^{\intercal}A_1x = 1\\
            & x \in \R^S\\
    \end{aligned}
    \tag{QCQP$_{\Delta}$}
\end{equation}
Here, $R^S = \{x \in \R^n : \supp(x) \subseteq S\}$.
The inner optimization problem can easily be solved using semidefinite programming methods for example, so the interesting question is in fact the outer optimization problem over $S$.

We now describe how the result of \Cref{thm:root_thm} arises naturally from the study of convex relaxations of \Cref{eq:sparse_qcqp}.

Various convex formulations of \Cref{eq:sparse_qcqp} have been considered, for example in \cite{atamturk2019rank, bach2010convex}.
For our purposes, we will consider the cone
\[
    \M(\Delta) = \conv \{xx^{\intercal} : x \in \R^n, \supp(x) \in \Delta\}.
\]
When $\Delta = \binom{[n]}{k}$, $\M(\Delta)$ is denoted $\FW^k_n$, and is refered to as the \emph{factor-width $k$} cone.
This cone has been studied extensively because of its connections to sparse quadratic programming \cite{boman2005factor, gouveia2022sums}.

In the one constraint case, it is not hard to see that \Cref{eq:sparse_qcqp} is equivalent to the following convex problem:
\begin{equation}
    \begin{aligned}
        \max\quad & \tr(A_0X)\\
        \st & \tr(A_1X) = 1\\
            & X \in \M(\Delta).
    \end{aligned}
    \label{eq:sparse_sdp}
\end{equation}
We think of $X$ as representing the matrix $x^{\intercal}x$, and because there is only one constraint, the optimum of \Cref{eq:sparse_sdp} (if it is feasible) must be of the form $x^{\intercal}x$ for some $x \in \R^n$ where $\supp(x) \in \Delta$.

The next step in defining this heuristic is to take the dual to \Cref{eq:sparse_sdp}
\begin{equation}\label{eq:sparse_sdp_dual}
    \begin{aligned}
        \min\quad & y\\
        \st & A_1y - A_0 \in \P(\Delta).
    \end{aligned}
\end{equation}
Here,
\[
    \P(\Delta) = \{X \in \Sym : \forall S \in \Delta,\;X|_S \succeq 0\}
\] is the conical dual to $\M(\Delta)$.
When $\Delta = \binom{[n]}{k}$, this cone and its connections to hyperbolicity cones were studied extensively in \cite{blekherman2022hyperbolic}, where it was denoted by $\mathcal{S}^{n,k}$.

We will assume that strong duality holds for this problem, which in the 1-constraint setting is equivalent to saying that $-A_1$ is not in $\mathcal{P}(\Delta)$.

The main observation is that if $X \in \P(\Delta)$, and $p$ is an LPM polynomial whose coefficients are nonnegative, then $p(X) \ge 0$, simply because determinants of PSD matrices are nonnegative.
Using this observation, we may think of $\P(\Delta)$ as being a kind of barrier function for $\P(\Delta)$, in that if $p(X) < 0$, then $X$ is not in $\P(\Delta)$.
From this, we can prove \Cref{thm:root_thm} theorem easily:
\begin{proof}[Proof of \Cref{thm:root_thm}]
    Because $A_1$ is positive definite, \Cref{eq:sparse_sdp} satisfies strong duality. Therefore, we consider the dual program \Cref{eq:sparse_sdp_dual}.

    Because $A_1$ is positive definite, for $y$ large enough, $A_1 y - A_0$ will be positive definite, and in particular will be in $\mathcal{P}(\Delta)$.
    By convexity, if there is some $y_0$ so that $A_1 y_0 - A_0$ is not in $\mathcal{P}(\Delta)$ then for all $y < y_0$, $A_1 y - A_0$ is not in $\mathcal{P}(\Delta)$.

    Suppose now that $p(A_1y_0 - A_0) = 0$. Then for some $S \in \Delta$, $\det((A_1y_0 - A_0)|_S) \le 0$.
    In particular, for any $y < y_0$, $(A_1 y - A_0)|_S$ cannot be positive semidefinite.
    Therefore, we have that the value of the dual program is at least $y_0$, and by strong duality, the value of the primal program must also be at least $y_0$.
\end{proof}

This theorem actually can be extended somewhat to give a continuous optimization problem that exactly recovers the value of \Cref{eq:sparse_qcqp} (though this problem will turn out to be intractible generally).
\begin{theorem}
    \label{thm:continuous_formulation}
    Consider a sparse QCQP $\mathcal{Q}$ as in \Cref{eq:sparse_qcqp} where $A_1$ is positive definite.
    Suppose that $p$ is an LPM polynomial supported on $\Delta$ with positive coefficients. For any diagonal matrix $D$, define the polynomial
    \[
        p_D(X) = p(DXD).
    \]
    Then $p_D$ is an LPM polynomial with positive coefficients supported on $\Delta$, and the value of $\mathcal{Q}$ is precisely
    \[
        \max_{D \in \text{Diag}} \eta_{g_D},
    \]
    where $g_D = p_D(A_1t - A_0)$.
\end{theorem}
We will defer the proof of this result to \Cref{sec:continuous}.

\Cref{thm:root_thm} is the starting point for a number of results.
In \Cref{subsec:heuristic}, we will give a method for efficiently finding a feasible $x$ whose value in \Cref{eq:sparse_qcqp} matches the one guaranteed by \Cref{thm:root_thm}.
In \Cref{sec:hyperbolic}, we will also describe a way to generalize theorem \ref{thm:root_thm} to cases with multiple constraints, and when $A_1$ is not positive definite.

\section{The Greedy Conditioning Heuristic}
\label{subsec:heuristic}
In this section, we give an efficient meta-algorithm for finding a feasible $x$ to \Cref{eq:sparse_qcqp} with value at least that guaranteed by \Cref{thm:root_thm}.
We will do this using a greedy algorithm and an idea which we call the \emph{conditioning trick}.
In fact, the $x$ we will produce will often be significantly better than what is guaranteed by a naive application of \Cref{thm:root_thm}.

We will assume here that $A_1$ is positive definite for this section.

Let $p = \sum_{S \in \Delta} a_S \det(X|_S)$ be an LPM polynomial.
Given a set $T \subseteq [n]$, we define the conditional polynomial $p|_T$ to be
\[
    p|_T = \sum_{T\subseteq S \in \Delta} a_S \det(X|_S).
\]
That is, rather than summing over all $S \in \Delta$, we only sum over those $S$ in $\Delta$ that also contain $T$.

We will also extend our earlier definition of $\eta_g$ to define $\eta_{p, A_1, A_0}$ to be the maximal real root of the polynomial $p(A_1 y - A_0)$, or $-\infty$ if there is no such real root.
When $A_1$ and $A_0$ are clear from context, or implicitly defined by a sparse QCQP, we will abuse notation and simply refer to  $\eta_p$ instead of $\eta_{p, A_1, A_0}$.

Using these definitions, we can state our greedy heuristic:
we first fix an LPM polynomial with nonnegative coefficients supported on $\Delta$.
\begin{algorithm}
    \caption{The Greedy Conditioning Heuristic}
    \label{alg:greedy}
    \begin{algorithmic}
        \State $T \gets \varnothing$
        \For{$t = 1 \dots k$}
            \State $j \gets \argmax \eta_{p|_{T + j}}$
            \State $T \gets T + j$
        \EndFor

        \Return T
    \end{algorithmic}
\end{algorithm}
Intuitively, $\eta_{p|_T}$ gives us a score for how well the sets in $\Delta$ containing $T$ perform in \Cref{eq:sparse_qcqp_2}.
In each round, we add an element to our current set that maximizes the marginal improvement in this score.

\begin{theorem}
    \label{thm:greedy_works}
    Fix some LPM polynomial $p$ of degree $k$ supported on $\Delta$.
    Suppose that we have an oracle that can evaluate $p$ at any symmetric matrix $X$ in exact arithmetic.
    We can then compute $\eta_{p|_{T}}$ for any $T \subseteq [n]$ using polynomially many arithmetic operations and evaluations of $p$.
    Furthermore, \Cref{alg:greedy} produces a set $T$ so that
    \begin{equation}
        \begin{aligned}
            \max\quad & x^{\intercal}A_0x\\
            \st & x^{\intercal}A_1x = 1\\
                & x \in \R^T\\
        \end{aligned}
        \ge \eta_p.
    \end{equation}
\end{theorem}
The precise time complexity of this algorithm depends on the amount of time required to compute the values of $\eta_{p|_T}$.

In the case when $p = c_n^k$ is a characteristic coefficient, we will implement a few more details in order to compute this algorithm particularly efficiently.
We give the details of this algorithm in \Cref{sec:char_poly}, but here, we will give an overview of the runtime 
\begin{theorem}
    \label{thm:characteristic}
    When $p = c_n^k$, and $A_1$ is positive definite, we can implement \Cref{alg:greedy} so that it computes a value which is at most $\eta_p$ using
    \[
        O(kn^3 + k^2n^{\omega})
    \]
    arithmetic operations.
\end{theorem}


\section{Sparse Linear Regression}
\label{sec:sparseReg}
Given a matrix $A \in \R^{n \times m}$, $b \in \R^m$ and $k \le n$, the sparse linear regression problem is to find
\[
    \min \{ \|A x - b\|_2^2 : \|\supp(x)\| \le k\}.
\]
In \cite{ben2022new}, it was shown that this regression problem can be cast as a sparse QCQP with one constraint. In particular, this $x$ is an optimum for the sparse linear regression problem if and only if it is optimal for the following sparse QCQP:

\begin{equation*}
\begin{aligned}
    \max\quad & x^{\intercal}(A^{\intercal}bb^{\intercal}A)x\\
    \st & x^{\intercal}A^{\intercal}Ax = 1\\
        & x \in \R^n\\
        &|\supp(x)| \le k.
\end{aligned}
\label{eq:sparse_reg}
\end{equation*}
It turns out that this particular equation has a particularly simple form that allows us to find a closed form solution for the maximum root.
\begin{theorem}
    \label{thm:sparse_reg_closed_form}
    Let $p$ be an LPM polynomial with nonnegative coefficients, and suppose that $p(A^{\intercal}A) \neq 0$.
    Then for the sparse regression problem in \Cref{eq:sparse_reg}, we have that
    \[
        \eta_p = \frac{p(A^{\intercal} ( I + bb^{\intercal})A)}{p(A^{\intercal}A)} - 1.
    \]
\end{theorem}
This closed form solution bares some resemblance to Cramer's rule for finding the solution to a system of linear equations; if $p$ is the determinant, then this in fact precisely recovers Cramer's rule for solving this regression problem.
It is important to note that computing this value does not require explicitly computing the roots of a polynomial; it only requires computing the value of this polynomial in two distinct points.
As a result of this, our algorithmic methods become much faster.

\begin{theorem}
    \label{thm:sparse_reg_fast}
    For sparse linear regression QCQPs, as defined in \Cref{eq:sparse_reg}, and when $p = c_n^k$, we can implement \Cref{alg:greedy} so that it requires $O(n^3 + kn^{\omega})$ arithmetic operations in exact arithmetic.
\end{theorem}

\subsection{A Probabilistic Interpretation}
\label{sec:probabilistic}
We will also give a probabilistic interpretation of this value of $\eta_p$ for sparse regression.
This interpretation can be used to give intuition about when $\eta_p$ is a good approximation for sparse regression.
We take a LPM polynomial $p = \sum_{S \in \Delta} a_S \det(X|_S)$ which is supported on a set $\Delta$.
Given a matrix $A$, we consider the following probability distribution over elements of $\Delta$: we define the probability of choosing $S \in \Delta$ to be
\[
    \Pr(S) = \frac{a_S \det(A^{\intercal}A|_S)}{p(A^{\intercal}A)}
\]
If $a_S = 1$ for all $S \in \binom{[n]}{k}$, then sampling from this distribution is known as voluming sampling \cite{deshpande2006matrix}.
This is also related to the theory of determinantal point processes and strongly Rayleigh distributions \cite{anari2016monte}.
These sampling methods are often used to produce low rank sketches of a matrix; our results can be viewed as providing a deterministic method for finding such a sketch.

Intuitively, this probability distribution weights the subsets of the columns of $A$ according to their diversity; the larger the volume that that subset of columns of $A$ span, the more likely they are to be selected.
The idea of using diversity as a prior for regression was also considered in statistics, as in \cite{kojima2016determinantal}.

We also define $\ell(A_S, b) = \max \{\|b\|^2 - \|A_Sx - b\|_2^2 : x \in \R^k\}$, the squared norm of the projection of $b$ onto the column space of $A$.
\begin{theorem}
    \label{thm:probabilistic_eq}
    If $p(A^{\intercal}A) \neq 0$, then
    \[
        \eta_{p|_T} = \EE[\ell(A_S, b) | T \subseteq S].
    \]
\end{theorem}
For this reason, we think of this heuristic in this case as being a diversity weighted version of sparse linear regression.
If we think that more diverse sets of columns of $A$ will be more effective on average than less diverse sets, then this method will produce better results.
In particular, we belive this method is not particularly effective when the underlying matrix $A$ is essentially random, as in that case, this diversity assumption does not hold.
The case when $A$ is a Gaussian random matrix is essentially the setting of the compressed sensing problem \cite{candes2005decoding}.

However, it seems that in real world data sets, diverse sets of columns often are preferable for regression, and so we will give experimental results for our heuristic on real world data sets in \Cref{subsec:expReg}.

\section{Sparse PCA}
\label{sec:sparsePCA}
The sparse PCA problem is to find the maximum sparse eigenvector of a given symmetric matrix, or formally, for a given symmetric matrix $A$, we define the maximum $k$-sparse eigenvalue of $A$ to be the value of the following program:
\begin{equation*}
\begin{aligned}
    \max\quad & x^{\intercal}Ax\\
    \st & x^{\intercal}x = 1\\
        & x \in \R^n\\
        &|\supp(x)| \le k.
\end{aligned}
\label{eq:sparse_pca}
\end{equation*}
For a given matrix $A$, we define $\lambda^{(k)}(A)$ to be the value of this program.

\begin{lemma}
    Let $c_n^k$ be the characteristic coefficient of $X$ of degree $k$, then the
    \[
        \lambda^{(k)}(A) \ge \max \{ t : c_n^k(t I - X) = 0\} \ge \lambda_k(A),
    \]
    where $\lambda_k(A)$ is the $k^{th}$ largest eigenvalue of the matrix $A$.
\end{lemma}
The fact that $\lambda^{(k)}(A) \ge \lambda_k(A)$ has been well known, since by the Cauchy interlacing theorem, for every $S \in \binom{[n]}{k}$, $\lambda_{max}(A|_S) \ge \lambda_k(A)$.
Interestingly though, this definition also is a basis invariant property of the matrix $A$, i.e. it only depends on the eigenvalues of $A$, since the polynomial $c_n^k$ only depends on the eigenvalues of its input.

We give an especially efficient algorithm for computing sparse PCA
\begin{theorem}
    \label{thm:sparse_pca_fast}
    For sparse PCA QCQPs, as defined in \Cref{eq:sparse_pca}, and when $p = c_n^k$, we can implement \Cref{alg:greedy} so that it requires $O(n^3 + k^2n^{\omega})$ arithmetic operations in exact arithmetic.
\end{theorem}

We will give experimental results for our heuristic on real world data sets in \Cref{subsec:expPCA}.

\section{Experimental Results}
\label{sec:experimental}
In this section, we will describe our experimental results for sparse PCA and sparse linear regression.

We will exclusively use the polynomial $p = c_n^k$ and the algorithm detailed in \Cref{sec:char_poly} for our experiments, and we will focus on the sparse linear regression problem defined in \Cref{eq:sparse_reg} and the sparse PCA problem defined in \Cref{eq:sparse_pca}.

\subsection{Some Notes on Implementation}
We implemented our methods in C++ using the Eigen linear algebra package \cite{eigenweb}. We ran our experiments on an 3.60GHz 4-Core Intel i7-9700K CPU with 16 GB of RAM.
Our code is available on github at \url{https://github.com/ootks/sparse_qcqps}.

\subsection{Experimental Results for Sparse Linear Regression}
\label{subsec:expReg}
We will evaluate our method in comparison to two existing methods, which we will describe briefly here:
\begin{itemize}
    \item LASSO\cite{tibshirani1996regression} - This minimizes the standard $L_2$ loss with an additional $L_1$ penalty added on, i.e. for some constant $\alpha$, it minimizes
    \[
        \|A x - b\|^2 + \alpha \|x\|_1.
    \]
    This is an extremely popular method for sparse linear regression.
    It is noteworthy that LASSO does not allow the user to directly specify a sparsity level, rather as $\alpha$ increases, solutions tend to become more sparse.
    Therefore, to use this as a feature selection algorithm, we first find $x$ minimizing this loss for a number of values of $\alpha$, sort the indices according to the size of their coefficients in $x$, and then consider the top $k$ indices for each $k$ to be the selected subset.
    We then evaluate performance by regressing $b$ using the columns of $A$ contained in $S$ and measuring the $L_2$ loss.
    \item Orthogonal Matching Pursuit\cite{tropp2004greed} - This is an alterative greedy method for performing sparse linear regression. As in \Cref{alg:greedy}, we construct a set $S$ by adding one element in each round.
    If $T$ is the set that has been constructed so far, OMP selects the next element to maximize
    \[
        \ell(A_{T+i}, b),
    \]
    and then projects all columns of $A$ onto the orthogonal complement of the column $A_i$, and also projects $b$ onto this orthogonal complement.
\end{itemize}
We test these three methods on 4 data sets, which we refer to as \emph{Communities}, \emph{Superconductivity}, \emph{Diabetes}, and \emph{Wine}.
These datasets can all be found on the UCI Machine Learning Repository\cite{Dua:2019}, except for Diabetes, which was found in \cite{efron2004least}.

We normalize so that each column has mean 0 and variance 1.
We will evaluate both the regression loss for each method for a number of different values of $k$, and also give the time required for each method to complete.

\begin{figure}[H]
    \centering
    \textbf{Communities}\par\medskip
    \begin{subfigure}[b]{0.4\textwidth}
        \includegraphics[width=\textwidth]{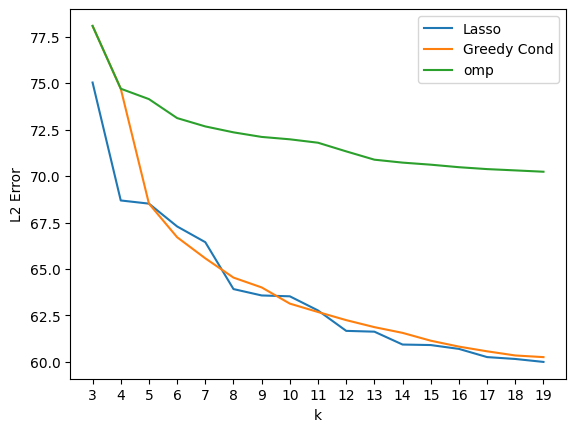}
        \caption{A plot of regression loss against $k$ for 3 methods. This data set had 101 features.}
    \end{subfigure}
    \begin{subfigure}[b]{0.4\textwidth}
        \begin{tabular}{c c c c}
            $k$ & Greedy Conditioning & OMP & LASSO \\
            \hline
            3 & 8 & 7 & 3\\
            6 & 12 & 17 & 24\\
            9 & 17 & 25 & 30\\
            12 & 21 & 33 & 33\\
            15 & 25 & 41 & 33\\
            17 & 28 & 41 & 33\\
        \end{tabular}
        \vspace{0.5in}
        \caption{Selected times in milliseconds. For LASSO, only the run that produced a set of a given size was timed.}
    \end{subfigure}
\end{figure}
\begin{figure}[H]
    \centering
    \textbf{Superconductivity}\par\medskip
    \begin{subfigure}[b]{0.4\textwidth}
        \includegraphics[width=\textwidth]{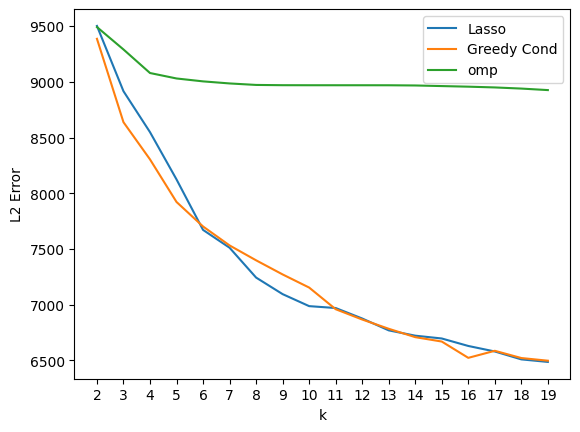}
        \caption{A plot of regression loss against $k$ for 3 methods. This dataset had 82 features.}
    \end{subfigure}
    \begin{subfigure}[b]{0.4\textwidth}
        \begin{tabular}{c c c c}
            $k$ & Greedy Conditioning & OMP & LASSO \\
            \hline
            3 & 24 & 76 & 30\\
            6 & 27 & 126 & 41\\
            9 & 29 & 171 & 49\\
            12 & 31 & 543 & 41\\
            15 & 34 & 349 & 96\\
            18 & 36 & 358 & 275\\
            \hline
        \end{tabular}
        \vspace{0.5in}
        \caption{Selected times in milliseconds. For LASSO, only the run that produced a set of a given size was timed.}
    \end{subfigure}
\end{figure}
\begin{figure}[H]
    \centering
    \textbf{Diabetes}\par\medskip
    \begin{subfigure}[b]{0.4\textwidth}
        \includegraphics[width=\textwidth]{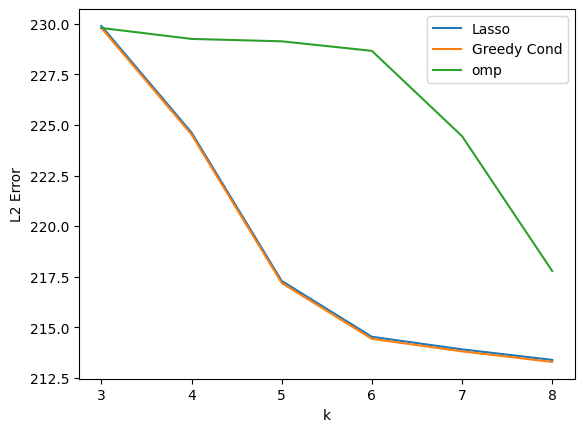}
        \caption{A plot of regression loss against $k$ for 3 methods. The dataset had 10 features. All reported times were under 1 millisecond.}
    \end{subfigure}
\end{figure}
\begin{figure}[H]
    \centering
    \textbf{Wine}\par\medskip
    \begin{subfigure}[b]{0.4\textwidth}
        \includegraphics[width=\textwidth]{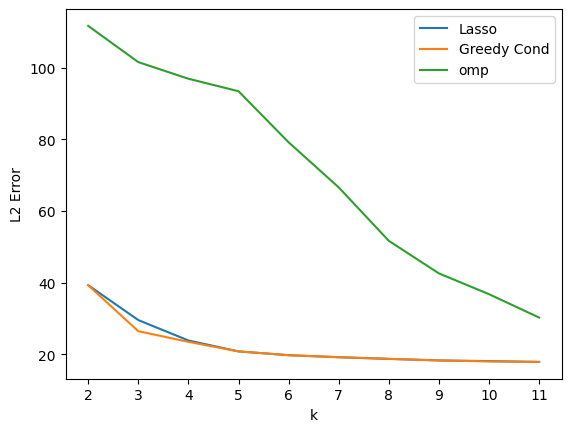}
        \caption{A plot of regression loss against $k$ for 3 methods. The dataset had 13 features. All reported times were under 1 millisecond.}
    \end{subfigure}
\end{figure}

\subsection{Experimental Results for Sparse PCA}
\label{subsec:expPCA}
For our experimental results, we follow the work done in \cite{bertsimas2022solving}, which gives exact values for a number of real world sparse PCA problems.
We will use the optimal solutions found in their paper to benchmark our results.
We consider 5 datasets, which all come from the UC Irvine Machine Learning Repository \cite{Dua:2019}: \emph{Wine}, \emph{Pitprops}, \emph{MiniBooNE}, \emph{Communities}, and \emph{Arrythmia}.

We find that our methods produce answers which achieve close to the optimal possible answer in most cases.

\begin{table}[H]
\begin{center}
    \begin{tabular}{c|c c c c c c}
        Dataset & Columns & $k$ & Found Value & Optimal Value & Gap & Time (s)\\
        \hline
        Wine & 13 & 5 & 3.43 & 3.43 & $<10^{-5}$ & $3\times 10^{-4}$\\
             &    & 10 & 4.45 & 4.59 & $0.03$ & $8\times 10^{-4}$\\
        \hline
        Pitprops & 13 & 5 & 3.40 & 3.40 & $<10^{-5}$ & $3\times 10^{-4}$\\
             &    & 10 & 3.95 & 4.17 & $0.05$ & $8\times 10^{-4}$\\
        \hline
        MiniBooNE & 50 & 5 & 4.99 & 5.00 & $<10^{-5}$ & 0.003\\
             &    & 10 & 9.99 & 9.99 & $<10^{-5}$ & 0.012\\
        \hline
        Communities & 101 & 5 & 4.51 & 4.86 & 0.07 & 0.02 \\
             &    & 10 & 8.71 & 8.82 & $0.013$ & 0.09\\
        \hline
        Arrythmia & 274 & 5 & 4.18 & 4.23 & 0.012 & 0.39\\
         & & 10 & 7.49 & 7.53 & 0.005 & 1.44
    \end{tabular}

\end{center}
\caption{A table describing the results of running an implementation of \Cref{alg:greedy} for sparse PCA on various datasets and values of $k$.  The gap is defined to be $\frac{\text{Optimal Value} - \text{Found Value}}{\text{Optimal Value}}$.}
\end{table}


\section{A Continuous Formulation of \Cref{eq:sparse_qcqp}}
\label{sec:continuous}
Our main goal in this section is to give a proof of \Cref{thm:continuous_formulation}
\begin{proof}
    If $p$ is an LPM polynomial supported on $\Delta$ with positive coefficients, then
    \[
        p(DXD) = \sum_{S \in \Delta} \left(\prod_{i \in S} D_{ii}^2\right) a_S \det(X|_S),
    \]
    which is clearly an LPM polynomial with nonnegative coefficients.
    Therefore, by \Cref{thm:root_thm}, we have that for any $D \in \text{Diag}$, $\eta_{p_D}$ is a lower bound on the value of $\mathcal{Q}$.

    It remains to show that for some diagonal matrix $D$, we have $\eta_{p_D}$ equals the value of $\mathcal{Q}$.

    For this, let $S$ be the optimal solution in the formulation of $\mathcal{Q}$ given in \Cref{eq:sparse_qcqp_2}, and let $D$ be the diagonal matrix so that $D_{ii} = 1$ if $i \in S$, and $0$ otherwise.
    Then, we see that
    \[
        p_D(A_1 t - A_0) = a_S\det((A_1t - A_0)|_S),
    \]
    and its maximum root is precisely the value of $t$ when $(A_1t - A_0)|_S$ is PSD and singular.
    That is, this is the minimum of the program
    \begin{equation}
        \begin{aligned}
            \min\quad & y\\
            \st & (A_1y - A_0)|_S \succeq 0.
        \end{aligned}
    \end{equation}
    We see that the dual of this program is precisely
    \begin{equation}
        \begin{aligned}
            \max\quad & \tr(A_0|_SX)\\
            \st & \tr(A_1|_SX) = 1\\
                & X \succeq 0\\
        \end{aligned},
    \end{equation}
    which has precisely the value of our original program by definition of $S$.
    Moreover, strong duality holds in this case because $A_1|_S$ is positive definite, so we are done.

    We conclude that the value of $\mathcal{Q}$ is also $\eta_{p_D}$, which shows the theorem.
\end{proof}

\section{Proofs for \Cref{subsec:heuristic} on the Greedy Conditioning Heuristic}
The goal of this section is to show that the greedy conditioning heuristic successfully finds some set $S$ which attains the value $\eta_p$ defined in \Cref{sec:lpm_relaxation}.
To do this, we will require some lemmas.
\begin{lemma}
    \label{lem:increasing}
    Assume that $A_1$ is positive definite.
    For any LPM polynomial $p$, and any $T \subseteq [n]$ so that there is some $S \in \Delta$ so that $T \subsetneq S$, there is some $i \in [n]$ so that
    \[
        \eta_{p|_{T}} \le \eta_{p|_{T + i}}.
    \]
\end{lemma}
\begin{proof}
    We have the following identity:
    \[
        (k - |T|) p|_T(X) = \sum_{i \in [n] \setminus T} p|_{T + i}(X).
    \]
    This can be seen by expanding out both polynomials in terms of minors of $X$ and comparing terms.

    Therefore, we have that 
    \[
        p|_T(A_1 \eta_{p|_T} - A_0) = 0 = \sum_{i \in [n] \setminus T} p|_{T + i}(A_1 \eta_{p|_T} - A_0).
    \]

    This implies that for some $i$, $p|_{T + i}(A_1 \eta_{p|_T} - A_0) \le 0$. Because $A_1$ is positive definite, $\lim_{t \rightarrow \infty}  p|_{T + i}(A_1 t - A_0) = \infty$, so by the intermediate value theorem, for some $t \ge \eta_{p|_T}$, $p|_{T + i}(A_1 t - A_0) = 0$, and therefore,
    \[
        \eta_{p|_{T + i}} \ge t \ge \eta_{p|_T}.
    \]
\end{proof}
To proceed, we will want two definitions and some technical lemmas.
\begin{defi}
    For $T \subseteq [n]$, let the \emph{Schur complement} of the matrix $X$ with respect to $T$ be
    \[
    X \setminus T = X|_{[n] - T} - X_{[n]-T,T} X|_T^{-1} X_{T,[n]-T},
    \]
    Here, $X_{[n]-T,T}$ denotes the nonprincipal submatrix of $X$ whose rows are indexed by $[n] - T$ and whose columns are indexed by $T$.
\end{defi}
\begin{defi}
    For an LPM polynomial $p$, define
    \[
        p_{-T}(X) = \sum_{S \in \Delta : T \subseteq S} a_S\det(X|_{S\setminus T}).
    \]
\end{defi}

\begin{lemma}
    \label{lem:schurCompCond}
    \[
        p|_T(X) = \det(X|_T) p_{-T}(X \setminus T).
    \]
\end{lemma}
\begin{proof}
    To prove this, we will need to recall the Schur complement determinant identity.
    \[
    \det(X|_T) \det(X \setminus T) = \det(X).
    \]

    We also crucially have the property that Schur complements commute with taking submatrices: if $T \subseteq S$, then
    \[
        (X \setminus T)|_{S \setminus T} = (X|_S)\setminus T.
    \]

    From this, we see that 
    \begin{align*} 
    \det(X|_T) p_{-T}(X \setminus T)&=\sum_{S \in \Delta : T \subseteq S} a_S \det(X|_T)\det((X \setminus T)|_{S\setminus T})\\
            &=\sum_{S \in \Delta : T \subseteq S} a_S \det(X|_S)\\
            &=p|_T(X),
    \end{align*}
    as desired.

\end{proof}
\begin{lemma}
    \label{lem:new_oracle}
    Fix some LPM polynomial $p$ of degree $k$.
    Suppose that we have an oracle that can evaluate $p$ at any symmetric matrix $X$ in exact arithmetic.
    Then we can compute the value of $p|_T$ for any matrix $X$ using at most $(k+1)$ calls to the oracle and $O(n^{\omega})$ additional arithmetic operations, where $\omega$ is the matrix multiplication constant.
\end{lemma}
\begin{proof}
    We refer to \Cref{lem:schurCompCond}.
    We can compute both $\det(X|_T)$ and $X \setminus T$ using at most $O(n^{\omega})$ arithmetic operations, so it remains to compute $p_{-T}(X)$ using at most $k+1$ oracle calls. To do this, notice that
    \[
        \frac{\partial }{\partial X_{ii}} \det(X|_S) = \begin{cases} 0 \text{ if }i \not \in S\\ \det(X|_{S - i}) \text{ if } i \in S \end{cases}.
    \]
    From this, and extending by linearity, we get that
    \[
    p_{-T}(X) =  (\sum_{i \in T} \frac{\partial}{\partial X_{ii}})^{|T|} p|_T(X) = D_{1_{T}}^{|T|} p|_T(X).
    \]
    Here, $D_{1_{T}}$ denotes the directional derivative with respect to the diagonal matrix $1_T$ whose $i^{th}$ diagonal entry is 1 if $i \in T$ and 0 otherwise.
    We now apply an alternative characterization of the directional derivative to obtain that
    \[
    D_{1_T}^{|T|} p|_T(X) = \frac{1}{|T|!}\frac{d}{dt}^{|T|}p(X + t 1_T)|_{t = 0}
    \]
    Notice that $p(X + t 1_T)$ is a univariate polynomial of degree at most $k$, and therefore, it can be specified by its $k+1$ coefficients.
    Using our oracle, we can compute this univariate polynomial in $k+1$ distinct points.
    Using these evaluations, and we can then apply polynomial interpolation using at most $O(k^{\omega})$ additional arithmetic operations.
    Once we have the $k+1$ coefficients of $p(X+t1_T)$, we can $|T|^{th}$ derivative at 0 by just taking its $|T|^{th}$ coefficient.
\end{proof}
We come to the proof of our main theorem.
\begin{proof}[Proof of \Cref{thm:greedy_works}]
    We use \Cref{lem:new_oracle} to produce an oracle for $p|_T(X)$.
    Then, we can compute the coefficients of the polynomial $g(t) = p|_T(A_1 t - A_0)$ evaluating this polynomial at $k+1$ locations using our oracle and polynomial interpolation.
    Once we have this, we can compute $\eta|_{p|_T}$ to arbitrary accuracy using polynomial root finding techniques \cite{pinkert1976exact,ben1988fast}.

    It is clear from \Cref{lem:increasing} that at round $t$ of the algorithm,
    \[
        \eta_{p|_T} \le \max_{j \in [n] \setminus T} \eta_{p|_{T+j}}.
    \]
    In particular, we have that $\eta_{p|_T}$ increases in every round of the algorithm, and in particular, it is always at least $\eta_p$, as desired.
\end{proof}

We will now consider a more detailed analysis for the characteristic coefficients in the next section.

\section{Characteristic Coefficients}
\label{sec:char_poly}
For this section, we recall that 
\[
    c_n^k(X) = \sum_{S \subseteq [n] : |S| = k} \det(X|_S).
\]
We will abbreviate $c_n^k$ by $p$ when convenient in this section.
The core idea that we will exploit to accelerate our implementation of \Cref{alg:greedy} is that while we must compute the characteristic coefficients of a large number of matrices in the course of the algorithm, these matrices are in fact closely related to each other.
In particular, we will mostly only need to compute the characteristic coefficients of matrices that differ from an earlier matrix by a rank 1 matrix.
It is often the case in numerical linear algebra that it is possible to update the values of a computation after a rank 1 update than it is to recompute the relevant values directly.

We will also refer to diagonalizations of matrices, that is a decomposition of a symmetric matrix $X$ as

\[
    X = Q\Lambda Q^{\intercal},
\]
where $Q$ is the orthogonal matrix whose columns are eigenvectors of $X$ and $\Lambda$ is the diagonal matrix of eigenvalues of $X$.
We will also discuss maintaining a diagonalization of a matrix under rank 1 updates.
While it is not possible to exactly compute a diagonalization of a matrix in rational arithmetic, we will ignore this point here, and assume that errors in eigenvalue computations when using floating point arithmetic are negiglible.
This is the typical assumption used in numerical linear algebra.
Given this assumption, it is well known that a diagonalization of a matrix can be computed in $O(n^3)$ floating point operations \cite{demmel1997applied}.

The ingredients for our technique will be the following: the idea of maintaining a diagonalization of a matrix using rank 1 updates; a fast method for computing $p|_i(X)$ for \emph{all} $i$ of $X$ simultaneously given a diagonalization of $X$, and Newton's method for root finding.

Firstly, it is known that if $X = Q\Lambda Q^{\intercal}$ is the diagonalization of $X$, then it is possible to compute a diagonalization of $X +\rho vv^{\intercal}$, where $v \in\R^n$ in $O(n^{\omega})$ floating point operations \cite{demmel1997applied,bunch1978rank}. 
Precisely,
\begin{lemma}
    Let $X = Q\Lambda Q^{\intercal}$ be an $n\times n$ symmetric matrix and a given diagonalization. Let $v \in \R^n$. Then there is an algorithm, \emph{UpdateDiagonalization}, to find $Q'$ and $\Lambda'$ so that 
    \[
        X + \rho vv^{\intercal} = Q'(\Lambda')(Q')^{\intercal}
    \]
    in $O(n^{\omega})$ floating point operations.
\end{lemma}
This technique, which is essentially finds roots of the characteristic polynomial of $X + vv^{\intercal}$ using a variant of Newton's method, is the basis of the divide and conquer method for diagonalizing tridiagonal matrices \cite{cuppen1980divide}.

Secondly, we will show the following theorem:
\begin{theorem}
    \label{thm:fastcomp}
    Suppose that $X = Q\Lambda Q^{\intercal}$ is a symmetric matrix with a given diagonalization. Then, there is an algorithm, \emph{Conditionals}, which compute the vector $\vec{v}$, where $\vec{v}_i = p|_{i}(X)$ for each $i$ using $O(n^{\omega})$ floating point operations.
\end{theorem}

Finally, we will note the following, which is essentially a statement about the number of iterations required for Newton's method to terminate.
\begin{lemma}
    \label{lem:newton}
    Let $g(t)$ be a univariate polynomial with only real roots, and let $r$ be the largest root of $g(t)$. Let $t_0$ be a computable upper bound on the maximum root of $g(t)$, then
    there is an algorithm \emph{MaxRoot} that we can compute $s$ satisfying $r \le s \le r + \epsilon$ in a number of arithmetic operations which is at most
    \[
        O(k^2\log(\frac{t_0-r}{\epsilon}))
    \]
\end{lemma}
For our purposes, we will assume that we are working in finite precision, and that for all of our applications, the numerical factor $\log(\frac{t_0-r}{\epsilon})$ in the previous discussion is in fact a constant. We will also note that if $g(t) = p|_T(tA_1 - A_0)$ and $A_1$ is PSD, then $g(t)$ is real rooted, as can be seen from the work done in \cite{blekherman2021linear}, so that we can apply this fact throughout.

We will also need the following two additional methods: a method \emph{Interpolate} for interpolating a univariate polynomial from $\ell$ evaluations of that polynomial in $O(\ell^2)$ operations, and a method \emph{SchurComplement($X$, $j$)}, which computes $X \setminus \{j\} = X - \frac{1}{X_{jj}} X_j X_j^{\intercal}$, which is clearly possible in $O(n^2)$ time.
\begin{remark}
    We will be somewhat loose about whether $X \setminus \{j\}$ is the rank one update of $X$ given by $X - \frac{1}{X_{jj}} X_j X_j^{\intercal}$, or the submatrix of this matrix obtained by deleting the $j^{th}$ row and column as above.
    The reason for this is that the characteristic coefficients of both matrices of the same degree are the same.
\end{remark}

Given these methods, we can state our faster heuristic algorithm in \Cref{alg:faster} when $A_1$ is PSD.
The main idea of \Cref{alg:faster} is to maintain a set of evaluation matrices, the $X^{(i)} = A_1t_i - A_0$, where we will evaluate our conditional polynomials.
The precise values of these evaluation matrices are theoretically not so important, so long as there are sufficiently many $t_i$'s, so that we can use these evaluations to interpolate the univariate polynomial $p|_T(A_1t-A_0)$.
In practice, it is useful to set the $t_i$'s to be the Chebyshev nodes \cite{berrut2004barycentric}, to avoid numerical instability.
As long as we maintain the diagonalizations of all of the $X^{(i)}$, we can evaluate all of the relevant polynomials at $X^{(i)}$ quickly.
In each round, we update the $X^{(i)}$ by taking their Schur complements with respect to the chosen column.
\begin{algorithm}[H]
    \caption{The Characteristic Method}
    \label{alg:faster}
    \begin{algorithmic}
        \Procedure{CharacteristicRoots}{$A_0$, $A_1$, $k$}
        \State Initialize distinct $t_1 ,\dots, t_{\ell} \in \R$ for evaluation.
        \State $T = \varnothing$
        \For{$i = 1 \dots \ell$}
            \State $X^{(i)} =  A_1t_i-A_0$
            \State $Q^{(i)}, \Lambda^{(i)} =$ Diagonalize$(X^{(i)})$
            \State detT$_i$ = 1
        \EndFor

        \For{$t = 1 \dots k$}
            \For{$i = 1 \dots \ell$}
            \State v$^{(i)}$ = Conditionals$(X^{(i)}, \Lambda^{(i)}, Q^{(i)})$.
            \EndFor
            \For{$j = 1 \dots n$}
            \State $g$ = Interpolate(detT$_1\cdot$v$^{(1)}_j$ ,\dots, detT$_{\ell}\cdot$v$^{(\ell)}_j$)
                \State $\eta_j = $MaxRoot($g$)
            \EndFor
            \State $j=\argmax \eta_j$.
            \State Add $j$ to $T$.
            \For{$i = 1 \dots \ell$}
                \State $Q^{(i)}, \Lambda^{(i)} =$ UpdateDiagonalization$(X^{(i)}, Q^{(i)}, \Lambda^{(i)}, X^{(i)}_j)$
                \State detT$_i$ = $X^{(i)}_{jj}$.
                \State $X^{(i)}$ = SchurComplement($X^{(i)}$, j)
            \EndFor
        \EndFor
        \Return $T$
        \EndProcedure
    \end{algorithmic}
\end{algorithm}
Here, we note if $A_0$ or $A_1$ has low rank, then in fact, it is possible to recover the polynomial $p(tA_0 - A_1)$ from fewer than $k$ evaluations.
For example, for sparse regression problems, we have seen that it is in fact possible to interpolate $p(tA_0 - A_1)$ using only 2 evaluations.
Hence, it is possible in some cases to implement this so that $\ell < k$.

Once we have this, it is not hard to compute the overall runtime of this algorithm:
\begin{theorem}
    The total runtime of \Cref{alg:faster} is $O(\ell n^3 + k\ell n^{\omega})$ operations. If $k = o(n^{3-\omega})$, then this is $O(\ell n^3)$ time.
\end{theorem}
\begin{proof}
    In the initialization phase of the algorithm, we need to diagonalize $\ell$ matrices, which takes a total of $\ell n^3$ time.
    In each subsequent round of the algorithm, of which there are $k$, the main contributions to the runtime complexity of the algorithm are in computing the conditionals of the $\ell$ values of $X^{(i)}$, which takes $O(\ell n^{\omega})$ time, and updating the diagonalizations of all $\ell$ of the $X^{(i)}$, which takes $O(\ell n^{\omega})$ time.

    Therefore, the total runtime complexity of this algorithm is $O(\ell n^3 + k \ell n^{\omega})$ operations.
\end{proof}
By noting that $\ell = 2$ for sparse regression, we obtain the following faster runtime:
\begin{cor}
    The total runtime of \Cref{alg:faster} for sparse regression is $O(n^3 + kn^{\omega})$ operations. If $k = o(n^{3-\omega})$, then this is $O(n^3)$ time.
\end{cor}
We also note that for sparse PCA, all of the matrices $X^{(i)} = t_i A_1 + A_0$ differ by multiples of the identity, it is in fact possible to do the intitial diagonalizations in $O(n^3)$ time total.
\begin{cor}
    The total runtime of \Cref{alg:faster} for sparse PCA is $O(n^3 + k^2n^{\omega})$ operations. If $k = o(n^{3-\omega})$, then this is $O(n^3)$ time.
\end{cor}

We will now go over the correctness of this algorithm and implementation of its subroutines.
\subsection{Correctness of \Cref{alg:faster}}
We will want a lemma to begin.
\begin{lemma}
    Let $p = c_n^k$, and let $T \subseteq [n]$ so that $|T| = t$, then
    \[
        p|_{T+j}(X) = \det(X|_T)c_{n-t}^{k-t}(X \setminus T)
    \]
\end{lemma}
\begin{proof}
    This follows immediately from \Cref{lem:schurCompCond}, and noting that $p_{-T} = c_{n-t}^{k-t}$.
\end{proof}
\begin{theorem}
    \Cref{alg:faster} correctly implements \Cref{alg:greedy}.
\end{theorem}
\begin{proof}
    To show that this correctly implements \Cref{alg:greedy}, we just need to show that the $\eta_j$ computed in each round of this algorithm is in fact $\eta_{p|_{T+j}}$.

    Clearly, because $\eta_j$ is the largest root of the interpolated polynomial given by
    \[
        g(t) = \text{Interpolate}(\text{detT}_1\cdot v^{(1)}_j,\dots, \text{detT}_{\ell}\cdot v^{(\ell)}_j),
    \]
    It suffices to show that 
    \[
        \text{detT}_i\cdot v^{(i)}_j = p|_{T+j}(t_iA_1 - A_0).
    \]

    We will show the following facts by induction on the round number:
    \begin{itemize}
        \item $X^{(i)} = (t_iA_1 - A_0)\setminus T$.
        \item $\text{detT}_i = \det((t_iA_1 - A_0)|_T)$.
        \item $\text{detT}_i\cdot v^{(i)}_j = p|_{T+j}(t_iA_1 - A_0)$.
    \end{itemize}
    The fact that this is the case in round 0 follows from the fact that $X^{(i)} = t_iA_1 - A_0$, $\text{detT}_i = 1$ and the definition of the Conditional method.

    Next, we see that at the end of each round, we take $T = T+j$, and $X^{(i)}$ becomes
    \[
        X^{(i)} \setminus \{j\}.
    \]
    By induction, we have that this is
    \[
        X^{(i)} \setminus \{j\} = ((t_iA_1 - A_0)\setminus T) \setminus \{j\} =  ((t_iA_1 - A_0)\setminus (T+j)),
    \]
    where this last identity is sometimes called the `sequential identity', as in \cite{zhang2006schur}, which states that repeated Schur complements compose.

    We also have that $\text{detT}_i$ is updated to be
    \begin{align*} 
    \text{detT}_i \cdot X^{(i)}_{jj}&=\det((t_iA_1 - A_0)\setminus T)((t_iA_1 - A_0)\setminus T)_{jj}\\
            &=\det((t_iA_1 - A_0)\setminus T)\det((t_iA_1 - A_0)|_{T + j}\setminus T)\\
            &=\det((t_iA_1 - A_0)\setminus (T+j)).
    \end{align*}
    Here, we have used the fact that determinants commute with the operation of taking a submatrix and the Schur cmplement determinant identity.

    Finally, we note that by definition,
    \[
        v^{(i)}_j = c^{k-t}_{n-t}|_{j}(X^{(i)}) = X_{jj} c^{k-t-1}_{n-t-1}(X^{(i)} \setminus \{j\}) =  X_{jj} c^{k-t-1}_{n-t-1}(X^{(i)} \setminus \{j\})
    \]
    We then apply our inductive hypothesis for $X^{(i)}$ and the compositionality of Schur complements to see that
    \begin{align*} 
    \text{detT}_i v^{(i)}_j&=\text{detT}_i X_{jj} c^{k-t-1}_{n-t-1}((t_iA_1 - A_0)\setminus (T+j))\\
                           &=\det((t_iA_1 - A_0)|_{(T+j)})p|_{T+j}(t_iA_1 - A_0)\\
                            &=p|_{T+j}(t_iA_1 - A_0),
    \end{align*}
    as desired.
\end{proof}

\subsection{Subroutines for \Cref{alg:faster}}
We want to show \Cref{thm:fastcomp}.
We will start by noting the following fact, which is the basis of our rank 1 update formula.
\begin{lemma}
    \label{lem:rank_one}
    Let $p = c^k_n$, then for any symmetric matrix $X$, with diagonalization $Q\Lambda Q^{\intercal}$
    \[
        p|_{i}(X) = X_{ii} c^{k}_n(\Lambda) - X_i^{\intercal}Q\nabla c^{k-1}_n(\Lambda)Q^{\intercal}X_i
    \]
    where $X_i$ denotes the $i^{th}$ column of $X$.
    Here, if $f : \Sym \rightarrow \R$, then we think of $\nabla f$ as being a matrix $G$ where $G_{ij} = \frac{\partial}{\partial X_{ij}}f(X)$.
\end{lemma}
\begin{proof}
    Let $X = Q\Lambda Q^{\intercal}$, where $Q$ is orthogonal and $\Lambda$ is diagonal, then by basis invariance,
    \[
        c^{k-1}_n(X - \frac{1}{X_{ii}}X_iX_i^{\intercal})= 
        c^{k-1}_n(\Lambda - \frac{1}{X_{ii}}Q^{\intercal}X_iX_i^{\intercal}Q).
    \]
    Note that $Q^{\intercal}X_iX_i^{\intercal}Q$ is rank 1, and in particular, $c^{k-1}_n$ vanishes to order $k-2$ at this point, so we have that the first order Taylor expansion of this polynomial is exact:

    \[
        c^{k-1}_n(\Lambda - \frac{1}{X_{ii}}Q^{\intercal}X_iX_i^{\intercal}Q) = 
        c^{k-1}_n(\Lambda) - \frac{1}{X_{ii}}X_i^{\intercal}Q\nabla c^{k-1}_n(\Lambda)Q^{\intercal}X_i
    \]
    In particular 
    \[
        p|_i(X) = X_{ii}c^{k-1}_n(\Lambda) - X_i^{\intercal}Q\nabla c^{k-1}_n(\Lambda)Q^{\intercal}X_i.
    \]
\end{proof}

For convenience of notation, let $e^k_n(x_1 ,\dots, x_n)$ denote the elementary symmetric polynomial
\[
    e^k_n(x_1 ,\dots, x_n) = \sum_{S \subseteq [n] : |S| = k} \prod_{i \in S}x_i,
\]
and for $x \in \R^n$, let $x_{\hat{i}} \in \R^{n-1}$ denote the vector obtained by deleting the $i^{th}$ entry of $x$.
\begin{lemma}
    Let $\Lambda$ be diagonal, and let $\lambda_i = \Lambda_{ii}$. Then $\nabla c^k_n(\Lambda)$ is a diagonal matrix so that 
    \[
        \nabla c^k_n(\Lambda)_{ii} = e^{k-1}_{n-1}(\lambda_{\hat{i}}).
    \]
\end{lemma}
\begin{proof}
    Let $\lambda \in \R^n$ be such that $\lambda_i = \Lambda_{ii}$ for each $i$. Let $\lambda_{\hat{i}} \in \R^{n-1}$ denote the vector obtained by deleting the $i^{th}$ entry of $\lambda$. 

    It is not hard to see that for $i \neq j$, and diagonal $\Lambda$,
    \[
        \frac{\partial}{\partial X_{ij}}\det(\Lambda) =
        \begin{cases}
            \prod_{k \neq i} \Lambda_{kk} \text{ if } i = j\\
            0 \text{ otherwise}
        \end{cases}.
    \]

    From this, we can tell from the definition of $c_k^n$ and linearity of the derivative that
    \[
        \frac{\partial}{\partial X_{ij}}c_k^n(\Lambda) = \begin{cases}e_{n-1}^{k-1}(\lambda_{\hat{i}}) \text{ if } i = j\\  0 \text{ otherwise }\end{cases}.
    \]
\end{proof}
We currently wish to compute the diagonal entries of $\nabla c^k_n(\Lambda)$ for a given diagonal matrix $\Lambda$.
From the above lemma, this is equivalent to computing $e^{k}_{n-1}(\lambda_{\hat{i}})$ for each $i \in [n]$.
We will give a dynamic programming algorithm for this computation, using the following obvious recurrence relation:
\[
    e_{n}^k(x) = x_ie_{n-1}^{k-1}(x_{\hat{i}}) + e_{n-1}^{k}(x_{\hat{i}}).
\]
For conceptual simplicity, we will first give an abstract lemma that encapsulates general kind of recurrence relation that we are interested in calculating.
\begin{lemma}
    \label{lem:recurrence}
    Let $(f_1)_{i=1}^n$ be a sequence of functions where $f_i : \R^i \rightarrow \R^k$. Suppose that $f_i$ is left invariant under permutations of its variables.
    Suppose further that for any $i \in [n]$, and any $j \in [i]$, there exists a function $T_i$ so that
    \[
        f_{i+1}(x) = T_i(f_i(x_{\hat{j}}), x_{j}),
    \]
    and that $T_i$ can be computed in time $O(k)$.
    Then for any $x \in \R^n$, the vector $f^* = (f_{n-1}(x_{\hat{i}}))_{i=1}^n$ can be computed in $O(kn\log(n))$ time.
\end{lemma}
\begin{proof}
    We first note that there is an easy to describe $O(kn^2)$ time algorithm for this computation.
    For any $x \in \R^{\ell}$, we can compute $f_{\ell}(x)$ by first computing $f_1(x_1)$, then $f_2(x_1,x_2) = T_1(f_1(x_1),x_2)$, then $f_3(x_1,x_2,x_3) = T_2(f_2(x_1,x_2), x_3)$, and so on, requiring a total of $\ell$ computations.
    So, given $x$, we can compute $f_{n-1}(x_{\hat{i}})$ in $O(n)$ steps for a given $i$, allowing us to compute $f^*$ in $O(kn^2)$ steps.

    This process can be seen to be wasteful; for example, we may recompute $f_i(x_1 ,\dots, x_i)$ multiple times, and we could instead only compute it once.
    It is clear then that dynamic programming techniques can be useful here, as they may allow us to save time recomputing values. 
    Thus, our goal is to minimize the number of distinct times we need to apply the recurrence.

    To simplify notation, we will fix $x$, and for $S = \{S_1 ,\dots, S_i\} \subseteq [n]$, denote by $f(S)$ the value of $f_{|S|}(x_{s_1},x_{s_2} ,\dots, x_{s_i})$. We then note that for any $i \in S$, we can compute $f(S)$, given the value of $f(S - i)$ in $O(k)$ time.

    We can now think of structuring our computation in the form of a tree. 
    Let $T$ be a labelled tree with vertices $V$ and edges $E$, with a distinguished root vertex $v_0$, and a labelling of the vertices, $\tau : V \setminus v_0 \rightarrow S$.
    For vertices $v, w \in V$, the distance from $v$ to $w$ is the number of edges on the unique path from $v$ to $w$ in $T$.
    The height of a vertex $V \in V$ is the distance of $v$ to $v_0$.
    For a given vertex $v$, let $P_v = \{v_0,v_1 ,\dots, v_{h}\}$ be the unique path from $v_0$ to $v_h = v$.
    We will say that $T$ has the rainbow path property if for each $v \in V$, the vertices in $P_v$ have distinct labels.

    We now claim that if we can compute a tree $T$ with the rainbow path property in $O(|V|)$ time, with the additional properties that that $T$ has $n$ leaves with different labels, all of height $n$, then we can compute $\nabla c^k_n(\Lambda)$ in $O(k|V|)$ time.

    To see this, we iterate over the vertices of $T$ in order of their height, using Breadth-First-Search starting from $v_0$. 
    For each vertex $v$, we compute $f(\tau(P_v))$, which we can do in time $O(k)$ because all vertices other than $v$ in $P_v$ have smaller height than $v$.
    Therefore, we can compute $f(\tau(P_v))$ for every $v \in T$ in time $O(k|V|)$ time.

    Because $T$ additionally has the rainbow path property, for any leaf $v$ of height $n$, $\tau(P_v - v)$ contains all elements of $[n]$ other than $\tau(v)$, so we can compute $f([n] - \tau(v))$ in constant time from the values computed in the tree.
    Therefore, after computing $f(\tau(P_v))$ for every $v$ in the tree, we can compute $f^*$ in $O(n)$ time, which gives a total running time of $O(kn\log(n))$ time.

    It remains to construct such a tree with $O(n\log(n))$ vertices in $O(n\log(n))$ time.

    To do this, for every $S \subseteq [n]$, we will define a tree $T(S)$ recursively.
    As a base case, suppose $S$ has 1 element, $s$. We define $T(S)$ to be the tree with $V = \{v_0,v_1\}$, $E = \{\{v_0,v_1\}\}$, and $\tau(v_1) = s$.

    Now, for a general set $S$, we divide $S$ into two disjoint subsets of nearly equal size: $L$ and $R$. That is, $L$ will contain $\ell = \lceil \frac{|S|}{2} \rceil$ elements and $R$ will contain $m = \lfloor \frac{|S|}{2} \rfloor$ elements.
    We then define the path graph $P_L$ on vertex set $\{w_1 ,\dots, w_{\ell}\}$, with edges from $w_i$ to $w_{i+1}$ for each $i$, and whose labelling labels the $w_i$ with distinct elements from $L$ in an arbitrary order.
    Similarly, define a path graph $P_R$ with vertex set $\{u_1 ,\dots, u_m\}$ with analogous edge set and labelling.

    Now, we construct $T(L)$ and $T(R)$ inductively.
    We then replace the root of $T(L)$ with $w_{\ell}$ (the last vertex of $P_L$) and similarly replace the root of $T(R)$ with $u_m$.
    Finally, we introduce a new root vertex $v_0$, and add edges from $v_0$ to $w_1$ and $u_1$.

    Let us see that $T(S)$ has the desired properties: the leaves of $T(S)$ are the union of the leaves of $T(L)$ and $T(R)$, so there are $n$ leaves of $T(S)$ which all have distinct labels by induction.
    We also have that the distance of a leaf of $T(L)$ to the root of $T(L)$ was $|L|$. The unique path from such a leaf of $T(L)$ to the root of $T(S)$ is now $|L| + |R| = |S|$.
    Similarly, the distance of a leaf of $T(R)$ to the root of $T(S)$ is also $|S|$. 

    Finally, we note that $T(S)$ has the rainbow path property: let $v \in T(L)$, so that the unique path from $v$ to the root of $T(S)$ also consists of the vertices on the path from $v$ to the root of $T(L)$ together with the vertices of $P_R$.
    By the rainbow path property of $T(L)$, the vertices on the path from $v$ to the last vertex of $P_R$ are distinctly labelled with elements from $L$ and the vertices of $P_R$ are distinctly labelled with elements from $R$, so clearly, all vertices on the path from $v$ to the root of $T(S)$ are distinctly labelled.
    
    Therefore, $T(S)$ has all of the desired properties

    Finally, note that $|T(S)| \le |T(L)| + |T(R)| + |S|$, and inductively, it can be seen that $|T([n])| \le \frac{n}{2}\log_2(\frac{n}{2}) +  \frac{n}{2}\log_2(\frac{n}{2}) + n \le n\log(n)$. Moreover, $T([n])$ can be constructed with a constant amount of work for each vertex, so our runtime requirement is also satisfied.

    The conclusion follows.
\end{proof}

\begin{lemma}
    Let $\lambda \in \R^n$. Then the vector $v \in \R^n$ where $v_i = e^{k}_{n-1}(\lambda_{\hat{i}})$ can be computed in $O(kn\log(n))$ arithmetic operations.
\end{lemma}
\begin{proof}
    We define
    \[
        f_i(x_1 ,\dots, x_i) = (e_i^1(x_1 ,\dots, x_i) ,\dots, e_i^k(x_1 ,\dots, x_i)),
    \]
    and define $T_i(v, x_{i+1})$ so that
    \[
        T_i(v, x)_j = xv_{j-1} + v_j.
    \]
    Then this pair satisfies hypotheses of \Cref{lem:recurrence}, and so we are done.
\end{proof}

As a note, \cite{baker1996computing} popularized the above recurrence relation for computing elementary symmetric polynomials, and it has been cited many times in the literature on determinantal point process.
While that paper also considers the problem of computing the elementary symmetric polynomials, it does not make use of this improved recurrence scheme to speed up the computation.

We will conclude with a proof of \Cref{thm:fastcomp}.
\begin{proof}[Proof of \Cref{thm:fastcomp}]
    We begin with the conclusion of \Cref{lem:rank_one}, that
    \[
        p|_{i}(X) = X_{ii} c^{k}_n(\Lambda) - X_i^{\intercal}Q\nabla c^{k-1}_n(\Lambda)Q^{\intercal}X_i
    \]

    Momentarily assume that $X$ is invertible, and note that 
    \[
        X^{-1} = Q\Lambda^{-1} Q^{\intercal}, 
    \]
    so 
    \[
        u_i = X^{-1}X_i = Q\Lambda^{-1} Q^{\intercal} X_i, 
    \]
    where $u_i$ is the $i^{th}$ standard basis vector.

    Therefore,
    \[
        Q^{\intercal} X_i = \Lambda Q^{\intercal} u_i
    \]

    So we can further simplify
    \[
        p|_i(X) = X_{ii}c^{k-1}_n(\Lambda) - u_i^{\intercal}Q\Lambda \nabla c^{k-1}_n(\Lambda)\Lambda Q^{\intercal} u_i.
    \]
    Notice that $u_i^{\intercal}Mu_i = M_{ii}$ for any matrix $M$, so we can vectorize this equation to say that 
    \[
        \vec{v} = c^{k-1}_n(\Lambda)\diag(X) - \diag(Q\Lambda \nabla c^{k-1}_n(\Lambda)\Lambda Q^{\intercal}),
    \]
    where $\diag(M)$ denotes the vector of diagonal entries of the matrix $M$.

    Now, note that 
    \[
        D = \Lambda \nabla c^{k-1}_n(\Lambda)\Lambda 
    \]
    is a diagonal matrix which we have seen we can compute in $O(kn\log(n))$ time.

    It is then not hard to see that for any diagonal matrix $D$, and any matrix $Q$,
    \[
        (QDQ^{\intercal})_{ii} = \sum_{j = 1}^n Q_{ij}^2 D_{jj} = (Q^{*2} \diag(D))_{ii},
    \]
    where $Q^{*2}$ is the entry-wise square of $Q$.
    Given $D$ and $Q$, this can clearly be computed in $O(n^{\omega})$ time.

    Therefore, the whole vector 
    \[
        \vec{v} = c^{k-1}_n(\Lambda)\diag(X) - Q^{*2} \diag(\Lambda^2  \nabla c^{k-1})
    \]
    can be computed in $O(n^{\omega} + kn\log(n))$ time.
\end{proof}

We will also show the following:
\begin{proof}[Proof of \Cref{lem:newton}]
    We will consider applying the Newton iteration by taking
    \[
        x_{n+1} = x_n - \frac{g(x_n)}{g'(x_n)},
    \]
    starting at $x_1 = t_0$.
    Each step requires $O(k)$ arithmetic operations to compute $g$, and we wish to argue that after $k\log(\frac{t_0-r}{\epsilon})$ steps, we will come with $\epsilon$ of the largest root.

    This method is clearly shift invariant, so for the analysis, we may assume that the largest root of $g$ is at 0. We then have that
    \[
        g(t) = t(t-r_1)(t - r_2)\dots(t-r_{k-1}),
    \]
    where $r_i \le 0$ are the real roots of $g$.

    Therefore, $g$ has nonnegative coefficients, and we see that for $t > 0$, $g(t), g'(t), g''(t) > 0$.
    In particular, $g(t)$ is convex when $t > 0$, and we have that
    \[
        x_n - \frac{g(x_n)}{g'(x_n)} \ge 0.
    \]
    for each $n$. Thus, the $x_n$ form a decreasing sequence of real numbers, which are bounded from below by 0.

    Now, using the product rule of derivatives, it can be seen that
    \[
        \frac{g(t)}{g'(t)} = \left( \frac{1}{t}+\sum_{i = 1}^{k-1} \frac{1}{t-r_i} \right)^{-1} \ge \frac{t}{k}
    \]
    Therefore,
    \[
       x_{n+1} \le \frac{k-1}{k}x_n.
    \]

    Therefore, after $k\log(\frac{t_0-r}{\epsilon})$ iterations, we will obtain that $x_{n} \le \epsilon$, as desired.
\end{proof}
\section{Proofs for \Cref{sec:sparseReg} on Sparse Linear Regression}
Our main result in this section is a proof of \Cref{thm:sparse_reg_closed_form}
\begin{proof}[Proof of \Cref{thm:sparse_reg_closed_form}]
    We consider the univariate polynomial
    \[p(A^{\intercal}A y - A^{\intercal}bb^{\intercal}A) = 0.\]

    Notice that when $y = 0$, we obtain the polynomial $p(A^{\intercal}bb^{\intercal}A)$. Now, notice that $X = A^{\intercal}bb^{\intercal}A$ is rank 1, and therefore, $\det(X|_S)$ vanishes to order $k-1$ at this point.
    Because $p$ is a linear combination of determinants, $p$ must then have a root of multiplicity at least $k-1$ at 0.

    Because $p(A^{\intercal}A) \neq 0$, and $A^{\intercal}A$ is positive semidefinite, we have that any root of this polynomial must be nonnegative, so we have that
    \[
        p(A^{\intercal}A y - A^{\intercal}bb^{\intercal}A) = y^{k-1}(ay - b) = ay^k-by^{k-1}
    \]
    for some $a,b\ge 0$. Hence, the maximal root of $p$ must be $\frac{b}{a}$.

    We can compute $a$ and $b$ explicitly. Notice that
    \[
        \lim_{y \rightarrow \infty} \frac{p(A^{\intercal}A y - A^{\intercal}bb^{\intercal}A)}{y^k} = p(A^{\intercal}A) = a,
    \]
    and that
    \[
        p(-1) = (-1)^kp(A^{\intercal}A y + A^{\intercal}bb^{\intercal}A) = (-1)^k(a+b)
    \]
    From this, we obtain that
    \[
        \frac{b}{a} = \frac{p(A^{\intercal}A y + A^{\intercal}bb^{\intercal}A)}{p(A^{\intercal}A)} - 1,
    \]
    as desired.
\end{proof}

We now show that this closed form solution is equivalent to the probabilistic result in \Cref{thm:probabilistic_eq}.
\begin{proof}[Proof of \Cref{thm:probabilistic_eq}]
    We first recall the so-called matrix determinant lemma, which states that for any invertible $X \in \R^{n\times n }$ and $v \in \R^{n}$
    \[
        \det(X + vv^{\intercal}) = (1+v^{\intercal}X^{-1}v)\det(X).
    \]

    This implies that
    \begin{align*}
        p|_{T}(A^{\intercal}A + A^{\intercal}bb^{\intercal}A) &= \sum_{S \in \Delta, T \subseteq S} a_S\det(A^{\intercal}A|_S + A^{\intercal}bb^{\intercal}A|_S)\\
                                                                   &= \sum_{S \in \Delta, T \subseteq S} a_S\det(A^{\intercal}A|_S)\big(1+b^{\intercal}A_S(A^{\intercal}A|_S)^{-1}A_S^{\intercal}b\big)\\
    \end{align*}

    We also recall the closed form formula for $\ell(A,b)$, given by
    \[
        \ell(A,b) = b^{\intercal}A(A^{\intercal}A)^{-1}A^{\intercal}b.
    \]

    We can thus simplify the above expression and see that

    \begin{align*}
        p|_{T}(A^{\intercal}A + A^{\intercal}bb^{\intercal}A) &= \sum_{S \in \Delta, T \subseteq S} a_S\det(A^{\intercal}A|_S)\big(1+\ell(A_S,b)\big)\\
                                                                   &= \left( \sum_{S \in \Delta, T \subseteq S} a_S\det(A^{\intercal}A) \right) + \sum_{S \in \Delta}\det(A^{\intercal}A|_S)\ell(A|_S,b)\\
                                                                   &= p_{\vec{a}}(A^{\intercal}A) + \sum_{S \in \Delta, T \subseteq S}a_S\det(A^{\intercal}A|_S)\ell(A_S,b)\\
    \end{align*}

    We then obtain the desired result:

    \begin{align*}
        \eta_{p|_T} = \frac{p|_{T}(A^{\intercal}A + A^{\intercal}bb^{\intercal}A)}{p|_{T}(A^{\intercal}A)} - 1 &= \big(\sum_{S \in \Delta, T \subseteq S}\Pr(S)\ell(A_S,b)\big)\\
                                                                                                      &= \EE[\ell(A_S, b) | T \subseteq S]
    \end{align*}
\end{proof}

\section{Proofs for \Cref{sec:sparsePCA} on Sparse PCA}
We show the following lemma:
\begin{lemma}
    Let $c_n^k$ be the characteristic coefficient of $X$ of degree $k$,
    \[
        \lambda^{(k)}(A) \ge \max \{ t : c_n^k(t I - X) = 0\} \ge \lambda_k(A),
    \]
    where $\lambda_k(A)$ is the $k^{th}$ largest eigenvalue of the matrix $A$.
\end{lemma}
\begin{proof}
    The first inequality follows easily from \Cref{thm:root_thm}. The inequality $\max \{ t : c_n^k(t I - X) = 0\} \ge \lambda_k(A)$ follows from root interlacing, which we will explain next.

    Let $\det(tI - X)$ be the characteristic polynomial of $X$, whose $k^{th}$ largest root is $\lambda_{k}(X)$. Then by applying the chain rule, we obtain that
    \[
        c_n^k(t I - X) = \frac{d^{n-k}}{dt^{n-k}}\det(tI - X).
    \]

    As we will discuss in the next section, the polynomial $c_n^k(tI-X)$ has only real zeros.
    Rolle's theorem implies that for a real rooted polynomial $g(t)$ with roots $r_1, \dots, r_d$, $g'(t)$ is real rooted with roots $s_1, \dots, s_{d-1}$ with the property that for $i \in [d-1]$
    \[
        r_i \le s_i \le r_{i+1}.
    \]
    See \cite{kummer2015hyperbolic} for more details.

    By induction, this implies that if the roots of $\frac{d^{n-k}}{dt^{n-k}}g(t)$ are $s_1, \dots, s_{k}$, then for $i \in [k]$,
    \[
        r_i \le s_i \le r_{n-k}.
    \]
    In this case, this implies that the largest root of $c_n^k(t I -X)$ is at least $\lambda_k(A)$.
\end{proof}
\section{Hyperbolic Polynomials}
\label{sec:hyperbolic}
We will review the theory of hyperbolic polynomials here.
An $n$-variate polynomial $p \in \R[x_1 ,\dots, x_n]$ is said to be \emph{hyperbolic} with respect to a vector $v \in \R^n$ if $p(v) > 0$ and for all $x \in \R^n$, all complex roots of the univariate polynomial $p_x(t) = p(x + tv)$ are real.
A basic example of interest for hyperbolic polynomials is the determinant of a symmetric matrix; if we take for $v$ the identity matrix $I \in \Sym$, then the spectral theorem implies that the polynomial $\det(X + tI)$ has real roots for any $X \in \Sym$.
This is equivalent to the determinant polynomial being hyperbolic with respect to the identity matrix.

Associated to any hyperbolic polynomial is its \emph{hyperbolicity cone}
\[
    \Lambda_v(p) = \{x \in \R^n : \forall t < 0, p(x + t v) \neq 0\}.
\]
Surprisingly, this set is convex for any hyperbolic polynomial. The fact that these cones are convex was first shown by G\"arding in \cite{gaarding1951linear} when studying differential equations.
Since then, hyperbolic polynomials have been studied intensely for their connections to computer science and combinatorics \cite{saunderson2019certifying, gurvits2007van, wagner2011multivariate}.

We say that an $n$-variate polynomial is stable if it is hyperbolic with respect to any $v$ in the nonnegative orthant.
We say that a polynomial is \emph{PSD-stable} if it is defined on $\Sym$; it is hyperbolic with respect to $I$, and $\Lambda_I(p)$ contains the PSD cone.

In \cite{blekherman2021linear}, it was shown that an LPM polynomial $p$ is PSD-stable if and only if the $n$-variate polynomial $p(\Diag(x_1 ,\dots, x_n))$ is stable, where $\Diag(x_1 ,\dots, x_n)$ is the diagonal matrix whose diagonal entries are $x_1 ,\dots, x_n$. In particular, we see that if $p(X)$ is PSD stable, then $p|_T(X)$ is PSD stable, since
\[
    p|_T(\Diag(x_1 ,\dots, x_n)) = (\prod_{i\in T}x_i)(\prod_{i\in T} \frac{\partial}{\partial x_i}) p|_T(\Diag(x_1 ,\dots, x_n)),
\]
and stable polynomials are closed under differentiation with respect to coordinate vectors and multiplication \cite{wagner2011multivariate}.

In particular, there exists an LPM polynomial which is PSD-stable and supported on $\Delta$ if and only if $\Delta$ is a \emph{hyperbolic matroid}, which are defined in \cite{choe2004homogeneous}.
It was also shown in \cite{blekherman2021linear} that for any PSD stable LPM polynomial supported on a set $\Delta$, $\Lambda_I(p)$ also contains $\mathcal{P}(\Delta)$.

Therefore, if we have a multivariate optimization problem of the form
\begin{equation}
    \begin{aligned}
        \min\quad & b^{\intercal} y\\
        \st & \sum_{i=1}^k A_i y_i - A_0 \in \P(\Delta),
    \end{aligned}
\end{equation}
then we can find a lower bound on this problem by considering, for any PSD-stable LPM polynomial $p$,
\begin{equation}\label{eq:sparse_lpm_dual}
    \begin{aligned}
        \min\quad & b^{\intercal} y\\
        \st & \sum_{i=1}^k A_i y_i - A_0 \in \Lambda_I(p).
    \end{aligned}
\end{equation}
This problem is tractible in the sense that $p$ serves as a self-concordant barrier function on $\Lambda_I(p)$, allowing for the usage of interior point methods to optimize as long as $p$ can be evaluated efficiently \cite{guler1997hyperbolic}.
However, there are currently no software implementations of hyperbolicity cone programming that are efficient in practice.
The main difficulty of implementing such a program at this level of generality is that $p$ usually has too many coefficients to specify precisely.

If the QCQP in question has only one constraint, then this idea allows us to extend the greedy conditioning procedure to cases when the constraint matrix $A_1$ is not positive definite.
However, even for nonsparse QCQPs with multiple constraints, it is a difficult question to understand how to recover a solution to the original QCQP from the the corresponding SDP relaxation.
It seems like it would be even harder to recover a sparse solution from our hyperbolicity cone relaxation.

For these reasons, for our main results, we only work in the 1 constraint case.

\subsection{Approximation Guarantees}
We will use the ideas contained in this section and in \cite{blekherman2022hyperbolic} to obtain some approximation guarantees for our methods.
These results will only bound the approximatation quality for $\eta_p$ for a single polynomial $p$, and thus, these results likely will not capture the quality of \Cref{alg:greedy}, which considers a large numebr of polynomials.
\begin{theorem}
    Let $\mathcal{Q}$ be a sparse QCQP with $A_1$ being positive definite with $\Delta = \binom{[n]}{k}$. Let $p = c_n^k$. Then the optimal value of $\mathcal{Q}$ is at most
    \[
        C_1\eta_{p} + C_2,
    \]
    where
    \[
        C_1 =\left(1+ \frac{(n-k)\tr(A_1)}{\lambda_{min}(A_1)n(k-1)}\right) \text{, and}
    \]
    \[
        C_2 = \frac{(n-k)\tr(A_0)}{\lambda_{min}(A_1)n(k-1)}.
    \]
\end{theorem}
\begin{proof}
    From our previous discussion, we have that
    \[
        A_1 \eta_p - A_0 \in \Lambda_I(p).
    \]
    In \cite{blekherman2022hyperbolic}, it was shown that for any $X \in \Lambda_I(p)$, $X + \frac{(n-k)\tr(X)}{n(k-1)}I \succeq 0$. Therefore,
    \[
        A_1 \eta_p - A_0 + \frac{(n-k)\tr(A_1 \eta_p + A_0)}{n(k-1)}I \succeq 0.
    \]
    Now, we have that $\frac{1}{\lambda_{min}(A_1)} A_1 \succeq I$, so we also know that
    \[
        A_1 \eta_p - A_0 + \frac{(n-k)\tr(A_1 \eta_p + A_0)}{\lambda_{min}(A_1)n(k-1)}A_1 \succeq 0.
    \]
    In particular,
    \begin{align*}
        y &= \eta_p  + \frac{(n-k)\tr(A_1 \eta_p + A_0)}{\lambda_{min}(A_1)n(k-1)}\\
         &=
        \left(1+ \frac{(n-k)\tr(A_1)}{\lambda_{min}(A_1)n(k-1)}\right)\eta_p  + \frac{(n-k)\tr(A_0)}{\lambda_{min}(A_1)n(k-1)}
    \end{align*}
    is a feasible point for \Cref{eq:sparse_sdp_dual}, and we have that $y$ is an upper bound on the optimal value of $\mathcal{Q}$.
\end{proof}
This result is somewhat crude, which is in part due to the fact that it only considers $\eta_p$, and not $\eta_{p|_T}$ for larger sets $T$.
Indeed, it can be seen that just using $\eta_p$ as a lower bound for the value of $\mathcal{Q}$ often producecs very bad results, and making use of $\eta_{p|_T}$ as in \Cref{alg:greedy} is necessary to produce interesting results.
Nevertheless, we make it explicit here as a starting point for more detailed analysis of \Cref{alg:greedy} in the future.

We can also do a crude analysis of when $\eta_p$ recovers the optimum value of \Cref{eq:sparse_qcqp_orig} exactly.
It is not hard to see from \Cref{thm:probabilistic_eq} that for sparse regression, this is only possible if, for every $S$ in the support of $p$, $S$ attains the optimum in \Cref{eq:sparse_qcqp_2}.
In fact, this is more generally true.
\begin{theorem}
    Let $p$ be a LPM polynomial with nonnegative coefficients whose support is $\Delta$.
    Let $\mathcal{Q}$ be a sparse QCQP of the form \Cref{eq:sparse_qcqp} with support $\Delta$, where $A_1$ is positive definite.
    $\eta_p$ equals the optimum value of $\mathcal{Q}$ if and only if for every $S \in \Delta$,
    \begin{equation}
    \begin{aligned}
        \max\quad & x^{\intercal}A_0x\\
        \st & x^{\intercal}A_1x = 1\\
            & x \in \R^S\\
    \end{aligned}
    \end{equation}
    equals the optimum value of $\mathcal{Q}$.
\end{theorem}
\begin{proof}
    It is the case that $\eta_p$ equals the optimum value of $\mathcal{Q}$ if and only if $\eta_p$ equals the optimum value of the dual program given in \Cref{eq:sparse_sdp_dual}.
    It can be seen that this happens if and only if
    \[
        A_1 \eta_p - A_0 \in \mathcal{P}(\Delta).
    \]
    If this occurs then we see that on the one hand, for every $S \in \Delta$,
    \[
        (A_1 \eta_p - A_0)|_S \succeq 0,
    \]
    and on the other,
    \[
        p(A_1 \eta_p - A_0) = \sum_{S \in \Delta} a_S\det((A_1 \eta_p - A_0)|_S) = 0.
    \]
    From this it follows that for all $S \in \Delta$,
    \[
        \det((A_1 \eta_p - A_0)|_S) = 0.
    \]

    Therefore, for all $S \in \Delta$, $(A_1 \eta_p - A_0)|_S$ is singular. Because $A_1$ is positive definite, we have that $\eta_p$ also solves the optimization problem
    \begin{equation}
        \begin{aligned}
            \min\quad & y\\
            \st & A_1|_Sy - A_0|_S \succeq 0.
        \end{aligned}
    \end{equation}

    We then see by strong duality that for every $S \in \Delta$,
    \begin{equation}
    \begin{aligned}
        \max\quad & x^{\intercal}A_0x\\
        \st & x^{\intercal}A_1x = 1\\
            & x \in \R^S\\
    \end{aligned}
    = \eta_p.
    \end{equation}
    This is then equal to the optimum value of $\mathcal{Q}$ by our assumption.
\end{proof}
We note that because \Cref{alg:greedy} considers a sequence of LPM polynomials, which terminates in an LPM with a singleton support, this result does not necessarily have implications for the exactness of \Cref{alg:greedy}.

\section{Conclusions and Future Directions}
Our results show that there is a fruitful connection between the roots of certain structured polynomials and sparse quadratic programming.
We have seen that these methods give an approach for a broad class of algorithmic problems, have connections to algebra, and also work well on real world optimization problems.

There are a number of ways that these results can be extended and improved.

\subsection{Different $\Delta$}
While we have seen that the set $\Delta = \binom{[n]}{k}$ has a very well behaved LPM polynomial supported on it, it is a challenge to find interesting other set families $\Delta$ for which such an LPM polynomial can be found.

Here is an example of a simple $\Delta$ which may appear in practice: say that $S_1,\dots, S_{\ell} \subseteq [n]$ are disjoint sets, and we want to ensure that at least one element from each set is chosen in our final output.
That is,
\[
    \Delta_{S_1, \dots, S_{\ell}} = \{S \subseteq [n] :|S| = k\;\forall i\in [\ell],\;S \cap S_i \neq \varnothing\}
\]
This may be of interest in practice if we have some prior information that each of these sets will contain different useful information, and so, we should take a selection from each of them.
We then define the following sequence of polynomials: we let $p_0 = c^n_k(X)$ and then define
\[
    p_i = p_{i-1}(X) - p_{i-1}(X|_{A_{i}^c}).
\]
It is not hard to see that this sequence of LPM polynomials has the desired property that it is supported exactly on $\Delta$, and that $p_{\ell}$ can be computed in time that is $O(2^{\ell}poly(n))$.
Thus, for a fixed $\ell$, this can be done efficiently, though it is unclear what to do if the number of sets grows with $n$.

In particular, we can consider a partition matroid: let $S_1, \dots, S_k$ be disjoint sets so that $S_1 \cup S_2 \cup \dots \cup S_k = [n]$ then the partition matroid for this set family is
\[
    \Delta_{S_1, \dots, S_k}  = \{S \subseteq [n] : |S \cap A_i| = 1\}.
\]
We want to understand why in the previous example, we are able to compute some associated LPM polynomial efficiently.

We define a notion of rank for a set family: the \emph{multiaffine rank} of a set family is the smallest $\ell$ so that there are diagonal matrices $D_1, \dots, D_{\ell}$ so that the LPM polynomial
\[
    \sum_{i=1}^{\ell} c_n^k(D_i X)
\]
has nonnegative coefficients and is supported on exactly $\Delta$.

It is clear that if a set family has multiaffine rank at most $\ell$, then there is an LPM polynomial supported on that set family which can be computed in $O(\ell k n^{\omega})$ time.
This is therefore an algebraic approach to understanding the complexity of computing such a polynomial.
The original example implies that for $\Delta_{S_1, \dots, S_{\ell}}$, the multiaffine rank is at most $2^{\ell}$.
We might ask if this is exactly the correct number.
\begin{question}
    What is the multiaffine rank of a partition matroid? In particular, is it $O(poly(n,k,\ell))$?
\end{question}

\subsection{Approximation Guarantees}
One major aspect that is missing from our analysis is a clear theoretical understanding of which instances these methods will perform well on.
For instance, in our experiments, we see that our methods do not tend to work well on random instances.
This is not surprising in light of \Cref{sec:probabilistic}: it seems that our methods implicitly impose some prior on which subsets are more likely to perform well on these tasks.
However, our experimental results indicate that these methods should perform at least as well as LASSO in many cases.
Therefore, we may ask
\begin{question}
    What structural properties of the matrices $A_1$ and $A_0$ lead to \Cref{alg:greedy} performing close to optimally relative to the true value of \Cref{eq:sparse_qcqp}?
\end{question}

\bibliographystyle{plain}
\bibliography{biblio.bib}
\end{document}